\documentclass[11pt]{article}
\usepackage[pdftex]{graphicx}
\usepackage{bm}
\usepackage{amsmath,amsthm,amssymb,url}
\usepackage{multirow,bigdelim}
\usepackage{amscd}

\newtheorem{thm}{Theorem}[section]
\newtheorem{theorem}[thm]{Theorem}
\newtheorem{proposition}[thm]{Proposition}
\newtheorem{lemma}[thm]{Lemma}

\newtheorem{corollary}[thm]{Corollary}

\DeclareMathOperator{\rank}{rank}

\renewcommand{\tilde}{\widetilde}

\usepackage{tikz}
\usetikzlibrary{arrows}

\title{Sufficient conditions for the global rigidity of periodic graphs}
\author{Vikt\'oria E. Kaszanitzky\thanks{
Budapest University of Technology and Economics, Magyar tud\'osok krt 2., Budapest, 1117, Hungary and the MTA-ELTE Egerv\'ary Research Group on Combinatorial Optimization, P\'azm\'any P\'eter s\'et\'any 1/C, Budapest, 1117, Hungary
(\texttt{kaszanitzky@cs.bme.hu}).
}
\and
 Csaba Kir\'aly\thanks{Department of Operations Research, ELTE E\"otv\"os Lor\'and University and the MTA-ELTE Egerv\'ary Research Group on Combinatorial Optimization, P\'azm\'any P\'eter s\'et\'any 1/C, Budapest, 1117, Hungary, ({\tt cskiraly@cs.elte.hu}).
}
\and
Bernd Schulze \thanks{
Department of Mathematics and Statistics,
Lancaster University,
Lancaster,
LA1 4YF, United Kingdom
(\texttt{b.schulze@lancaster.ac.uk}).}
}

\begin{document}


\maketitle

\begin{abstract}

Tanigawa (2016) showed that vertex-redundant rigidity of a graph implies its global
rigidity in arbitrary dimension. We extend this result to periodic graphs under fixed
lattice representations. A periodic graph is vertex-redundantly rigid if the deletion of a single
vertex orbit under the periodicity results in a periodically rigid graph. Our proof is
similar to the one of Tanigawa, but there are some added difficulties. First,
it is not known whether periodic global rigidity is a generic property. This issue is
resolved via a slight modification of a recent result of Kaszanitzy, Schulze and Tanigawa (2016).
Secondly, while the rigidity of finite graphs in $\mathbb{R}^d$ on at most $d$ vertices
obviously implies their global rigidity, it is non-trivial to prove a similar result for 
periodic graphs. This is accomplished by extending a result of Bezdek and Connelly (2002) 
on the existence of a continuous movement between two equivalent $d$-dimensional realisations
of a single graph in $\mathbb{R}^{2d}$ to periodic frameworks.

As an application of our result, we give a necessary and sufficient condition for the
global rigidity of generic periodic body-bar frameworks in arbitrary dimension. This
provides a periodic counterpart to a result of Connelly, Jord\'{a}n and Whiteley (2013) 
regarding the global rigidity of generic finite body-bar frameworks.

%
%
\end{abstract}



\section{Introduction}
\label{sec:intro}

A $d$-dimensional bar-joint framework is a pair $(G,p)$, where $G$ is a simple graph and $p$ is a map which assigns a point in $\mathbb{R}^d$ to each vertex of $G$. We think of $(G,p)$ as a straight line realisation of $G$ in $\mathbb{R}^d$, where the edge lengths are measured by  the standard Euclidean metric. Loosely speaking, $(G,p)$ is \emph{rigid} if any edge-length preserving continuous motion of the vertices of $(G,p)$ is necessarily a congruent motion (i.e., a motion corresponding to an isometry of $\mathbb{R}^d$). Moreover, $(G,p)$ is \emph{globally rigid} if it is the only framework in $d$-space with the same graph and edge lengths, up to congruent motions. It is well known that both rigidity and global rigidity are generic properties, in the sense that a generic realisation of a graph $G$ in $\mathbb{R}^d$  is rigid (globally rigid) if and only if \emph{every} generic realisation of $G$ in $\mathbb{R}^d$  is rigid (globally rigid) \cite{asiroth,gluck,gortler2010}. Therefore, a graph $G$ is called rigid (globally rigid) if some (equivalently any) generic realisation of $G$ is rigid (globally rigid).

The celebrated Laman's theorem from 1970 gives a combinatorial characterisation of the rigid graphs in  $\mathbb{R}^2$ \cite{Lamanbib}. Extending this result to higher dimensions is a fundamental open problem in distance geometry \cite{W1}. Similarly, a combinatorial characterisation of the globally rigid graphs in $\mathbb{R}^2$ has been obtained by Jackson and Jord\'an in 2005 \cite{jj}, but the problem of extending this result to higher dimensions also remains open. For the special class of body-bar frameworks \cite{W1}, however, complete combinatorial characterisations for rigidity and global rigidity have been established in all dimensions in \cite{Tay} and \cite{cjw}, respectively.

Tanigawa recently proved the following result, which is an important new tool to investigate the global rigidity of frameworks in $\mathbb{R}^d$.

\begin{theorem}[\cite{T1}] \label{thm2rig} Let $G$ be a rigid graph in $\mathbb{R}^d$ and suppose $G-v$ remains rigid for every vertex $v$ of $G$. Then $G$ is globally rigid in $\mathbb{R}^d$.
\end{theorem}

In particular, the following combinatorial characterisation of globally rigid body-bar frameworks in $\mathbb{R}^d$ by Connelly, Jord\'an and Whiteley \cite{cjw} easily follows from this result.

\begin{theorem}[\cite{cjw,T1}]\label{thmbbglobal} A generic body-bar framework is globally rigid in  $\mathbb{R}^d$ if and only if it is rigid in $\mathbb{R}^d$ and it remains rigid after the removal of any edge.
\end{theorem}

In Sections \ref{sec:2rigid} and \ref{sec:bbglob}, we obtain analogues of these results for infinite periodic frameworks under fixed lattice representations. Due to their applications in fields such as crystallography, materials science, and engineering, the rigidity and flexibility of periodic structures has seen an increased interest in recent years (see e.g. \cite{bs10,BST12,kst,mt13,rossd,ross2}). In particular,
 combinatorial characterisations of  generic rigid and globally rigid periodic bar-joint frameworks under fixed lattice representations in $\mathbb{R}^2$ were obtained in \cite{ross2} and \cite{kst}, respectively. Analogous to the situation for finite frameworks, extensions of these results to higher dimensions remain key open problems in the field.

For the special class of periodic body-bar frameworks, however, Ross gave a  combinatorial characterisation for generic rigidity in $\mathbb{R}^3$ \cite{rossbb}, and this result was recently extended to all dimensions by Tanigawa in \cite{T2} (see also Theorem~\ref{thm:bodybarrig}).
As an application of the main result of Section \ref{sec:2rigid} (Theorem~\ref{thm:2RGR}), we give the first combinatorial characterisation of generic \emph{globally rigid}  periodic body-bar frameworks  in all dimensions in Section~\ref{sec:bbglob} (Theorem~\ref{thm:bodybarglobalrig}).


\section{Preliminaries}
\label{sec:pre}
\subsection{$\Gamma$-labelled graphs and periodic graphs} \label{subsec:gain}

Let $\Gamma$ be a group isomorphic to $\mathbb{Z}^k$ for some integer \(k>0\).
A {\em $\Gamma$-labelled graph} is a pair $(G,\psi)$ of a finite directed (multi-) graph $G$ and a map $\psi:E(G)\rightarrow \Gamma$. 

For a given $\Gamma$-labelled graph $(G,\psi)$, one may construct a \emph{$k$-periodic  graph}
$\tilde{G}$ by setting $V(\tilde{G})=\{\gamma v_i: v_i\in V(G), \gamma\in \Gamma\}$
and $E(\tilde{G})=\{\{\gamma v_i, \psi(v_iv_j) \gamma v_j\}: (v_i, v_j)\in E(G), \gamma \in \Gamma\}$. This $\tilde{G}$ is called the {\em covering} of $(G, \psi)$ and $\Gamma$ is the {\em periodicity} of $\tilde{G}$. The graph $(G,\psi)$ is also called the {\em quotient $\Gamma$-labelled graph} of $\tilde{G}$.

To guarantee that the covering of $(G, \psi)$ is a simple graph, we assume that $(G,\psi)$ has no parallel edges with the same label when oriented in the same direction. Moreover, we assume that $(G,\psi)$ has no loops. This is because a loop in $(G,\psi)$ (with a non-trivial label) does not give rise to any constraint when we study the rigidity and flexibility of the covering $\tilde{G}$ under fixed lattice representations, as will become clear below.


Note that the orientation of $(G,\psi)$ is only used as a reference orientation and may be changed, provided that we also modify $\psi$ so that if an edge has a label $\gamma$ in one direction, then it has the label $\gamma^{-1}$ in the other direction. The resulting  $\Gamma$-labelled graph still has the same covering $\tilde{G}$.

It is also often useful to modify $(G,\psi)$ by using the switching operation. A {\em switching} at $v\in V(G)$  by $\gamma\in \Gamma$ changes $\psi$ to $\psi'$
defined by $\psi'(e)=\gamma \psi(e)$ if $e$ is directed from $v$, $\psi'(e)=\gamma^{-1} \psi(e)$ if $e$ is directed to $v$, and $\psi'(e)=\psi(e)$ otherwise. It is easy to see that a switching operation performed on a vertex in $(G,\psi)$ does not alter the covering $\tilde{G}$, up to isomorphism. 

Given a $\Gamma$-labelled graph $(G,\psi)$, we define a \emph{walk} in $(G,\psi)$ as an alternating sequence $v_1,e_1,v_2\ldots, e_k,v_{k+1}$  of vertices and edges such that $v_i$ and $v_{i+1}$ are the endvertices of $e_i$. For a  closed walk $C=v_1,e_1,v_2,\ldots, e_k,v_1$ in $(G,\psi)$, let $\psi(C)=\prod_{i=1}^k\psi(e_i)^{\textrm{sign}(e_i)}$, where $\textrm{sign}(e_i)=1$ if $e_i$ has forward direction in $C$, and $\textrm{sign}(e_i)=-1$ otherwise.
For a subgraph $H$ of $G$ define
$\Gamma_H$ as the subgroup of $\Gamma$ generated by the elements $\psi(C)$, where $C$ ranges  over all closed walks in $H$.
The {\em rank} of $H$ is defined to be the rank of $\Gamma_H$.
Note that the rank of $G$ may be less than the rank of $\Gamma$, in which case the covering graph $\tilde{G}$ contains an infinite number of connected components.



\subsection{Periodic bar-joint frameworks}
Recall that a pair $(G,p)$ of a simple graph $G=(V,E)$ and a map
$p:V\rightarrow \mathbb{R}^d$ is called a {\em (bar-joint) framework} in $\mathbb{R}^d$.
A periodic framework is a special type of infinite framework defined as follows.

Let $\tilde{G}=(\tilde{V},\tilde{E})$ be a $k$-periodic graph with periodicity $\Gamma$,
and let  $L:\Gamma \rightarrow \mathbb{R}^d$ be a nonsingular homomorphism with $k\leq d$,
where $L$ is said to be nonsingular if $L(\Gamma)$ has rank $k$.
A pair $(\tilde{G},\tilde{p})$ of $\tilde{G}$ and $\tilde{p}:\tilde{V}\rightarrow \mathbb{R}^d$ is said to be an
{\em $L$-periodic framework} in $\mathbb{R}^d$ if  \begin{equation}
\label{eq:periodic_func}
\tilde{p}(v)+L(\gamma)=\tilde{p}(\gamma v) \qquad \text{for all } \gamma\in \Gamma \text{ and all } v\in \tilde{V}.
\end{equation}
We also say that a pair $(\tilde{G}, \tilde{p})$ is {\em $k$-periodic} in $\mathbb{R}^d$ if it is $L$-periodic for some nonsingular homomorphism $L:\Gamma\rightarrow \mathbb{R}^d$.
Note that the rank $k$ of the periodicity may be smaller than $d$.

An $L$-periodic framework $(\tilde{G}, \tilde{p})$ is {\em generic} if the set of coordinates is algebraically independent over the rationals modulo the ideal generated by the equations (\ref{eq:periodic_func}).

A {\em $\Gamma$-labelled framework} is defined to be a triple $(G,\psi, p)$ of a finite $\Gamma$-labelled graph $(G,\psi)$ and a map $p:V(G)\rightarrow \mathbb{R}^d$.
Given a nonsingular homomorphism $L:\Gamma\rightarrow \mathbb{R}^d$, the {\em covering} of $(G,\psi,p)$ is the $L$-periodic framework $(\tilde{G}, \tilde{p})$,
where $\tilde{G}$ is the covering of $G$ and $\tilde{p}$ is uniquely determined from $p$ by (\ref{eq:periodic_func}). $(G,\psi,p)$ is also called the {\em quotient $\Gamma$-labelled framework} of $(\tilde{G}, \tilde{p})$.

We say that a $\Gamma$-labelled framework $(G,\psi,p)$ is {\em generic} if  the set of coordinates in \(p\) is algebraically independent over the rationals.
Note that an $L$-periodic framework $(\tilde{G}, \tilde{p})$ is generic if and only if the quotient $(G,\psi,p)$ of $(\tilde{G}, \tilde{p})$ is generic.

\subsection{Periodic body-bar frameworks}\label{subsec:bb}

A $d$-dimensional \emph{body-bar framework}  consists of disjoint full-dimensional rigid bodies in $\mathbb{R}^d$ connected by disjoint bars, and may be considered as a special type of bar-joint framework, as we will describe below. The rigidity and flexibility of body-bar frameworks has been studied extensively (see e.g. \cite{cjw, rossbb, Tay, W1}), as they have important applications in fields such as engineering, robotics, materials science, and  biology. The underlying graph of a body-bar framework is a multi-graph $H=(V(H),E(H))$ with no loops, where each vertex in $V(H)$ corresponds to a rigid body, and each edge in $E(H)$ corresponds to a rigid bar. To represent a body-bar framework as a bar-joint framework, we extract the  \emph{body-bar graph} $G_H$ from the multi-graph $H$ as follows (see also \cite{T1}, for example). $G_H$ is the simple graph with vertex set $V_H$ and edge set $E_H$, where
\begin{itemize}
\item $V_H$ is the disjoint union of vertex sets $B_H^v$ for each $v\in V(H)$, with $B_H^v$ defined as $B_H^v=\{v_1,v_2, \ldots, v_{d+1}\} \cup \{v_e |\, e \in E(H) \textrm{ is incident to } v\}$;
\item $E_H= \big(\bigcup_{v\in V(H)} K(B_H^v) \big) \cup \{e'=u_ev_e|\, e=uv \in E(H)\}$, where $K(B_H^v)$ is the complete graph on $B_H^v$.
\end{itemize}
For each $v\in V(H)$, the vertices of $B_H^v$ induce a complete subgraph of $G_H$, which is referred to as the \emph{body associated with $v$}. A bar-joint framework $(G_H,p)$ with $p:V_H \to \mathbb{R}^d$ is called a \emph{body-bar realisation} of $H$ in $\mathbb{R}^d$. See Figure 1 for an example.

 \begin{figure}[h!]\label{example}
  \includegraphics[width=15cm]{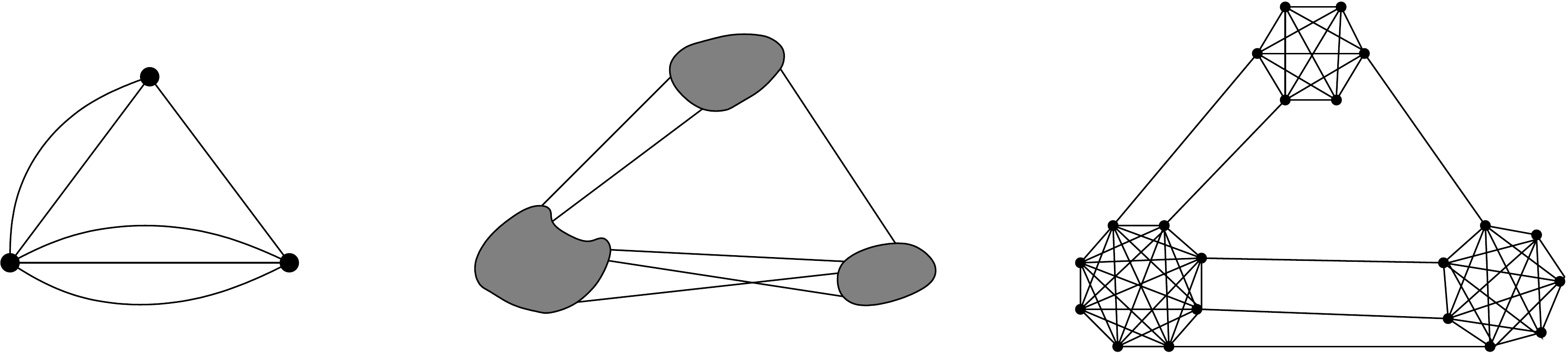}
  \caption{Example of a 2-dimensional multi-graph \(H\) (on the left) which is the underlying graph of the body-bar framework in the middle. On the right the graph \(G_H\) is shown.}
 \end{figure}


To define a periodic body-bar framework, we start with a $\Gamma$-labelled graph $(H,\psi)$, as defined in Section~\ref{subsec:gain}. However, we now allow $(H,\psi)$ to have loops with non-trivial labels, as well as parallel edges with equal labels when oriented in the same direction. Thus, $(H,\psi)$ defines  a \emph{$k$-periodic multi-graph} $\tilde{H}$ which has no loops but may have parallel edges. We now use the procedure described above to construct the $k$-periodic body-bar graph $G_{\tilde H}$ from the multi-graph $\tilde{H}$, with the slight modification that for any edge $e\in E(\tilde{H})$ joining a vertex $v$ with $\gamma v$ for some $\gamma \neq \textrm{id}$, we add two vertices $v_{e^-}$ and $v_{e^+}$ (instead of just one vertex $v_e$) to $B^v_{\tilde{H}}$,  and define $e'$ to be the edge $v_{e^-} \gamma v_{e^+}$ (instead of $v_{e} \gamma v_{e}$). This guarantees that the quotient $\Gamma$-labelled graph of the body-bar graph $G_{\tilde H}$ has no loops. 

 An $L$-periodic bar-joint framework $(G_{\tilde{H}},\tilde p)$ with $\tilde p:V_{\tilde{H}}\to \mathbb{R}^d$ is called an \emph{$L$-periodic body-bar realisation} of $\tilde H$ in $\mathbb{R}^d$.


\subsection{Rigidity and global rigidity}
Let $G=(V,E)$ be a graph.
Two bar-joint frameworks $(G,p)$ and $(G,q)$ in $\mathbb{R}^d$ are said to be \emph{equivalent} if
\begin{equation*}
\label{eq:periodic_equiv}
\left\| p(u)-p(v)\right\|=\left\| q(u)-q(v)\right\|\qquad \text{for all } uv\in E.
\end{equation*}
They are \emph{congruent} if
\begin{equation*}
\label{eq:periodic_congr}
\left\| p(u)-p(v)\right\|=\left\| q(u)- q(v)\right\|\qquad \text{for all } u,v\in V.
\end{equation*}
A bar-joint framework $(G,p)$ is called {\em globally rigid} if every framework $(G,q)$ in $\mathbb{R}^d$ which is equivalent to $(G,p)$ is also congruent to $(G,p)$.

Analogously, following \cite{kst}, we define an $L$-periodic bar-joint framework $(\tilde{G}, \tilde{p})$  in $\mathbb{R}^d$ to be {\em $L$-periodically globally rigid} if every $L$-periodic framework in $\mathbb{R}^d$ which is equivalent to $(\tilde{G}, \tilde{p})$ is also congruent to $(\tilde{G}, \tilde{p})$. Note that if the rank of the periodicity is equal to zero, then $L$-periodic global rigidity coincides with the global rigidity of finite frameworks.

A key notion to analyse $L$-periodic global rigidity is $L$-periodic {\em  rigidity}.  A framework $(\tilde{G},\tilde{p})$ is called {\em $L$-periodically rigid} if there is an open neighborhood $N$ of $\tilde{p}$ in which  every $L$-periodic framework $(\tilde{G},\tilde{q})$ which is equivalent to $(\tilde{G},\tilde{p})$ is also congruent to $(\tilde{G},\tilde{p})$. If $(\tilde{G},\tilde{p})$ is not $L$-periodically rigid, then it is called {\em $L$-periodically flexible}.


 A bar-joint framework $(\tilde{G},\tilde{p})$ is called {\em $L$-periodically vertex-redundantly rigid}, or {\em $L$-periodically 2-rigid} in short, if for every vertex orbit \(\tilde v\) of \(\tilde{G}\), the framework $(\tilde{G}-\tilde{v},\tilde{p}|_{V(\tilde{G})-\tilde{v}})$ is $L$-periodically rigid.

\subsection{Characterisation of $L$-periodic rigidity}

A key tool to analyse the rigidity or global rigidity of finite frameworks is the length-squared function and its Jacobian, called the rigidity matrix.
We may use the same approach to analyse periodic rigidity or periodic global rigidity (see also \cite{kst}).

For a $\Gamma$-labelled graph $(G,\psi)$ and $L:\Gamma\rightarrow \mathbb{R}^d$, define
$f_{G,L}:\mathbb{R}^{d|V(G)|}\rightarrow \mathbb{R}^{|E(G)|}$ by
\[f_{G,L}(p)=(\dots,\left\|p(v_i)-(p(v_j)+L(\psi(v_iv_j)))\right\|^2,\dots) \qquad (p\in \mathbb{R}^{d|V(G)|}).\]
For a finite set $V$, the {\em complete $\Gamma$-labelled graph} $K(V,\Gamma)$ on $V$ is defined
to be the graph on $V$ with the edge set $\{(u, \gamma v): u,v \in V, \gamma\in \Gamma\}$.
We simply denote $f_{K(V,\Gamma), L}$ by $f_{V,L}$.

By (\ref{eq:periodic_func}) we have the following fundamental fact.
\begin{proposition}
\label{prop:trivial}
Let $(\tilde{G}, \tilde{p})$ be an $L$-periodic framework and let $(G, \psi, p)$ be a quotient $\Gamma$-labelled framework of $(\tilde{G}, \tilde{p})$. Then $(\tilde{G}, \tilde{p})$ is $L$-periodically globally rigid (resp.~rigid) if and only if for every $q\in \mathbb{R}^{d|V(G)|}$ (resp.~for every $q$ in an open neighborhood of $p$ in $\mathbb{R}^{d|V(G)|}$), $f_{G,L}(p)=f_{G,L}(q)$ implies $f_{V(G),L}(p)=f_{V(G), L}(q)$.
\end{proposition}
We may therefore say that a $\Gamma$-labelled framework $(G,\psi,p)$ is {\em $L$-periodically globally rigid (or rigid)} if for every $q\in \mathbb{R}^{d|V(G)|}$ (resp. for every $q$ in an open neighborhood of $p$ in $\mathbb{R}^{d|V(G)|}$), $f_{G,L}(p)=f_{G,L}(q)$ implies $f_{V(G),L}(p)=f_{V(G), L}(q)$,
and we may focus on characterising the $L$-periodic global rigidity (or rigidity) of $\Gamma$-labelled frameworks.  If  $(G,\psi,p)$ is not $L$-periodically rigid, then it is called {\em $L$-periodically flexible}. A $\Gamma$-labelled framework $(G,\psi,p)$ is {\em $L$-periodically $2$-rigid} if for every vertex $v$ of $G$, the $\Gamma$-labelled framework $(G-v,\psi|_{G-v},p|_{V(G)-v})$ is $L$-periodically rigid.


We have the following basic result for analysing $L$-periodic rigidity.
\begin{theorem}[\cite{rossd}, \cite{kst}]
\label{prop:rigidity_matrix}
Let $(G,\psi, p)$ be a  generic $\Gamma$-labelled framework in $\mathbb{R}^d$
with $|V(G)|\geq d+1$ and rank $k$ periodicity $\Gamma$, and let $L:\Gamma\rightarrow \mathbb{R}^d$ be nonsingular.
Then $(G,\psi,p)$ is $L$-periodically rigid if and only if
\[
\rank df_{G,L}|_p=d|V(G)|-d-{d-k\choose 2},
\]
where $df_{G,L}|_p$ denotes the Jacobian of $f_{G,L}$ at $p$.
\end{theorem}

For combinatorial characterisations of generic $L$-periodically rigid or globally rigid $\Gamma$-labelled frameworks in $\mathbb{R}^2$,  we refer the reader to \cite{ross2,kst} and \cite{kst}, respectively.
A combinatorial characterisation of generic $L$-periodically rigid body-bar frameworks in $\mathbb{R}^d$ has been established in \cite{T2} (see also Theorem~\ref{thm:bodybarrig}).

\section{Rigidity implies global rigidity for small graphs}
\label{sec:rigidglob}

We first prove the following periodic counterpart of \cite[Lemma 1]{kneser}.

\begin{lemma}\label{lem:genkneser}
Let $p$ and $q$ be two $\mathbb{Z}^k$-periodic realisations of $n$ points $v_1,\dots, v_n$ in $\mathbb{R}^d$ with the same lattice
$L:\mathbb{Z}^k\rightarrow \mathbb{R}^d$. Furthermore, let $p(\gamma v_i)=L(\gamma)+p(v_i)=p_{\gamma, i}$ and $q(\gamma v_i)=L(\gamma)+q(v_i)=q_{\gamma, i}$ for $i=1,\dots, n$ and $\gamma\in \mathbb{Z}^k$. Let $p_{\gamma, i}:[0,1]\to \mathbb{R}^{2d}$ be the following continuous maps for $i=1,\dots, n$:
\begin{equation}
p_{\gamma,i}(t)=\left(\frac{p_{\gamma,i}+q_{\gamma,i}}{2}+(\cos (\pi t))\frac{p_{\gamma,i}-q_{\gamma,i}}{2},(\sin (\pi t))\frac{p_{\gamma,i}-q_{\gamma,i}}{2}\right).
\end{equation}
Then $p_{\gamma,i}(0)=(p_{\gamma,i},0^d)$ and $p_{\gamma,i}(1)=(q_{\gamma,i},0^d)$, where $0^d$ denotes the $d$-dimensional zero vector. Further, $|p_{\gamma,i}(t)-p_{\gamma',j}(t)|$ is monotone and $p_{\gamma,i}(t)=p_{0,i}(t)+(L(\gamma),0^d)$ for every $i,j\in\{1,\dots,n\}$ and $\gamma,\gamma'\in \mathbb{Z}^k$.
\end{lemma}
\begin{proof}
We only prove the last equation as the other statements follow directly from \cite[Lemma 1]{kneser}. Observe that
\begin{eqnarray*}
p_{\gamma,i}(t)&=&\left(\frac{p_{\gamma,i}+q_{\gamma,i}}{2}+(\cos (\pi t))\frac{p_{\gamma,i}-q_{\gamma,i}}{2},(\sin (\pi t))\frac{p_{\gamma,i}-q_{\gamma,i}}{2}\right)\\&=&
\left(\frac{p_{0,i}+L(\gamma)+q_{0,i}+L(\gamma)}{2}+(\cos (\pi t))\frac{p_{0,i}+L(\gamma)-(q_{0,i}+L(\gamma))}{2},\right.\\& &\left. (\sin (\pi t))\frac{p_{0,i}+L(\gamma)-(q_{0,i}+L(\gamma))}{2}\right)
\\&=&\left(\frac{p_{0,i}+q_{0,i}}{2}+L(\gamma)+(\cos (\pi t))\frac{p_{0,i}-q_{0,i}}{2},(\sin (\pi t))\frac{p_{0,i}-q_{0,i}}{2}\right)\\&=&
p_{0,i}(t)+(L(\gamma),0^d)
\end{eqnarray*}
holds for every  $i\in\{1,\dots,n\}$ and $\gamma\in \mathbb{Z}^k$.
\end{proof}

Lemma \ref{lem:genkneser} implies the following theorem.

\begin{theorem}\label{thm:nonglobnonrig}
Let $L:\Gamma\rightarrow \mathbb{R}^d$ be a lattice and let $(G,\psi,p)$ be a $\Gamma$-labelled framework in $\mathbb{R}^d$ which is not $L$-periodically globally rigid. Then the framework $(G,\psi,(p,0^d))$ is  $(L,0^d)$-periodically flexible in $\mathbb{R}^{2d}$, where $(L,0^d):\Gamma\rightarrow \mathbb{R}^{2d}$ maps $\gamma\in \Gamma$ to $(L(\gamma),0^d)$.
\end{theorem}
\begin{proof}
Let $(G,\psi,q)$ be an $L$-periodic framework which is equivalent but not congruent to $(G,\psi,p)$. By Lemma \ref{lem:genkneser}, there exists a continuous deformation between these two frameworks in $\mathbb{R}^{2d}$ that maintains the lattice and, by the monotonicity of the distances, also maintains the bar lengths. Therefore, this map proves that $(G,\psi,(p,0^d))$ is  $(L,0^d)$-periodically flexible.
\end{proof}

Let $(G,\psi,p)$ be an $L$-periodic framework in $\mathbb{R}^d$ with rank $k$ periodicity and with $|V(G)|\leq d-k+1$. Observe that $(G,\psi,q)$ in $\mathbb{R}^D$ has affine span of dimension at most $|V(G)|+k-1\leq d$ for $D\geq d$ with $(L,0^{D-d})$-periodicity. Now suppose that $(G,\psi,p)$ is $L$-periodically rigid in $\mathbb{R}^d$. Then it also has to be $L$-periodically globally rigid in $\mathbb{R}^d$ as during its continuous motion in $\mathbb{R}^{2d}$ the framework spans an at most \(d\)-dimensional subspace. Hence we have the following corollary of Theorem \ref{thm:nonglobnonrig}.

\begin{corollary}\label{cor:smallrigidglobrigid}
Let $(G,\psi,p)$ be a  $\Gamma$-labelled framework in $\mathbb{R}^d$ with rank $k$ periodicity and $L:\Gamma\rightarrow \mathbb{R}^d$. Suppose that $(G,\psi,p)$ is  $L$-periodically rigid and $|V(G)|\leq d-k+1$. Then $(G,\psi,p)$ is also $L$-periodically globally rigid.
\end{corollary}


\section{2-Rigidity implies global rigidity}
\label{sec:2rigid}

In this section we extend Theorem~\ref{thm2rig} to periodic frameworks by showing that $L$-periodic $2$-rigidity, together with a rank condition on the $\Gamma$-labelled graph in the case when the framework is $d$-periodic in $\mathbb{R}^d$, implies $L$-periodic global rigidity. We need the following lemmas.

\begin{lemma}[\cite{jjt}]
\label{lem:pq}
Let $f:\mathbb{R}^d\rightarrow \mathbb{R}^k$ be a polynomial map with rational coefficients
and $p$ be a generic point in $\mathbb{R}^d$.
Suppose that $df|_p$ is nonsingular.
Then for every $q\in f^{-1}(p)$  we have $\overline{\mathbb{Q}(p)}=\overline{\mathbb{Q}(q)}$.
\end{lemma}

Let $\Gamma$ be a group isomorphic to $\mathbb{Z}^k$, $t=\max\{d-k, 1\}$,
$(G,\psi)$ be  a $\Gamma$-labelled graph with $|V(G)|\geq t$, and  $L:\Gamma\rightarrow \mathbb{R}^d$ be nonsingular.
We pick any $t$ vertices $v_1,\dots, v_t$, and define the augmented function of $f_{G,L}$ by $\hat{f}_{G,L}:=(f_{G,L}, g)$, where $g:\mathbb{R}^{d|V(G)|}\rightarrow \mathbb{R}^{d+{d-k\choose2}}$ is a rational polynomial map
given by
\[
g(p)=(p_1(v_1), \dots, p_d(v_1),
p_1(v_2), \dots, p_{d-1}(v_2),\dots, p_1(v_t), \dots,  p_{d-t+1}(v_t))\qquad (p\in \mathbb{R}^{d|V(G)|})
\]
 where $p_i(v_j)$ denotes the $i$-th coordinate of $p(v_j)$.
Augmenting $f_{G,L}$ by $g$ corresponds to ``pinning down'' some coordinates to eliminate trivial continuous motions.

\begin{lemma}[\cite{kst}]
\label{prop:rigidity_matrix2}
Let $(G,\psi,p)$ be a  $\Gamma$-labelled framework in $\mathbb{R}^d$ with rank $k$ periodicity
and $L:\Gamma\rightarrow \mathbb{R}^d$.
Suppose that $p$ is generic and $|V(G)|\geq \max\{d-k, 1\}$.
Then
\[
\rank d\hat{f}_{G,L}|_p=\rank df_{G,L}|_p+d+{d-k\choose 2}.
\]
\end{lemma}

We also need an adapted version of \cite[Lemma 4.5]{kst}. To state this lemma, we require the following definition. 

Let $(G,\psi)$ be a $\Gamma$-labelled graph and let $v$ be a vertex of $G$.
Suppose that every edge incident to $v$ is directed from $v$.
For each pair of nonparallel edges $e_1=vu$ and $e_2=vw$ in $(G,\psi)$,
let $e_1\cdot e_2$ be the edge from $u$ to $w$ with  label \(\psi(vu)^{-1}\psi(vw)\).
We define \((G_v,\psi_v)\) to be the $\Gamma$-labelled graph obtained from $(G,\psi)$ by removing $v$
and inserting $e_1\cdot e_2$ (unless it is already present) for every pair of nonparallel edges $e_1, e_2$ incident to $v$.

\begin{lemma} \label{lem:redrig}
Let $(G,\psi,p)$ be a generic $\Gamma$-labelled framework in $\mathbb{R}^d$  with
rank $k$ periodicity $\Gamma$ and with $|V(G)|\geq d-k+1$
and let $L:\Gamma\rightarrow \mathbb{R}^d$ be nonsingular. Suppose that the covering $(\tilde G,\tilde p)$ has a vertex $v$ with at least $d+1$ neighbours \(\gamma_0v_0,\gamma_1v_1,\dots,\gamma_dv_d\), where $v,v_i\in V(G),\gamma_i\in \Gamma$, so that $\tilde p(\gamma_0v_0), \tilde p(\gamma_1v_1),\ldots \tilde p(\gamma_dv_d)$ affinely span \(\mathbb{R}^d\). Suppose further that
\begin{itemize}
\item $(G-v,\psi|_{G-v}, p')$ is $L$-periodically rigid in $\mathbb{R}^d$, with notation \(p'=p|_{V(G)-v}\), and
\item $(G_v, \psi_v, p')$ is $L$-periodically globally rigid in $\mathbb{R}^d$.
\end{itemize}
Then $(G, \psi, p)$ is $L$-periodically globally rigid in $\mathbb{R}^d$.
\end{lemma}
\begin{proof}
Pin the framework $(G,\psi,p)$ and take any $q\in \hat{f}_{G,L}^{-1}(\hat{f}_{G,L}(p))$.  Since $|V(G)|\geq d-k+1>\max\{d-k, 1\}$, we may assume that $v$ is not ``pinned'' (i.e., $v$ is different from the vertices selected when augmenting $f_{G,L}$ to $\hat{f}_{G,L}$). Our goal is to show that $p=q$.

Let $p'$ and $q'$ be the restrictions of $p$ and $q$ to $V(G)-v$, respectively.
Since $(G-v,\psi|_{G-v},p')$ is $L$-periodically rigid,
$d\hat{f}_{G-v,L}|_{p'}$ is nonsingular by Theorem~\ref{prop:rigidity_matrix} and Lemma~\ref{prop:rigidity_matrix2}.
Hence it follows from Lemma~\ref{lem:pq} that $\overline{\mathbb{Q}(p')}=\overline{\mathbb{Q}(q')}$.
This in turn implies that $q'$ is generic.

Consider the edges $e_0=vv_0, e_1=vv_1,\dots, e_{d}=vv_d$ in $(G,\psi)$ (all assumed to be directed from $v$) with respective labels $\psi(e_0)=\gamma_0, \psi(e_1)=\gamma_1,\ldots, \psi(e_d)=\gamma_d$.  Note that we may have $v_i=v_j$ for some $i,j$. By switching, we may further assume that $\gamma_0=\textrm{id}$.
For each $1\leq i\leq d$, let
\begin{align*}
x_i&=p(v_i)+L(\gamma_i)-p(v_0), \\
y_i&=q(v_i)+L(\gamma_i)-q(v_0),
\end{align*}
and let $P$ and $Q$ be the $d\times d$-matrices whose $i$-th row is $x_i$ and $y_i$, respectively.
Note that since $p(v_i)+L(\gamma_i)-p(v_0)=\tilde{p}(\gamma_iv_i)-\tilde{p}(v_0)$, and $q(v_i)+L(\gamma_i)-q(v_0)=\tilde{q}(\gamma_iv_i)-\tilde{q}(v_0)$, and 
 $p'$, \(q'\) are generic, $x_1,\dots, x_d$ and $y_1,\dots, y_d$ are, respectively, linearly independent, and hence $P$ and $Q$ are both nonsingular.

Let $x_v=p(v)-p(v_0)$ and $y_v=q(v)-q(v_0)$.
We then have $\|x_v\|=\|y_v\|$ since $G$ has the edge $vv_0$ with $\psi(vv_0)=\textrm{id}$.
Due to the existence of the edge $e_i$ we also have
\begin{align*}0 &=
\langle p(v_i)+L(\gamma_i)-p(v), p(v_i)+L(\gamma_i)-p(v)\rangle  -
\langle q(v_i)+L(\gamma_i)-q(v),q(v_i)+L(\gamma_i)-q(v) \rangle\\
&=\langle x_i-x_v, x_i-x_v\rangle - \langle y_i-y_v, y_i-y_v\rangle\\
&=(\|x_i\|^2-\|y_i\|^2)-2\langle x_i, x_v\rangle +2\langle y_i, y_v\rangle,
 \end{align*}
 where we used $\|x_v\|=\|y_v\|$.
Denoting by $\delta$ the $d$-dimensional vector whose $i$-th coordinate is equal to $\|x_i\|^2-\|y_i\|^2$,
the above $d$ equations can be summarized as
$$0=\delta-2P^T x_v+2Q^T y_v$$ which is equivalent to
$$y_v=(Q^T)^{-1}P^T  x_v -\frac{1}{2}(Q^T)^{-1}\delta.$$
By putting this into  $\|x_v\|^2=\|y_v\|^2$, we obtain
\begin{equation}
\label{eq:algebraic}
x_v^T(I_d-PQ^{-1}(PQ^{-1})^T)x_v-(\delta^T Q^{-1}(Q^{-1})^T P^T)x_v+\frac{1}{4}\delta^T Q^{-1}(Q^{-1})^T\delta=0,
\end{equation}
where $I_d$ denotes the $d\times d$ identity matrix.

Note that each entry of $P$ is contained in $\mathbb{Q}(p')$, and each entry of $Q$ is contained in $\mathbb{Q}(q')$. Since $\overline{\mathbb{Q}(p')}=\overline{\mathbb{Q}(q')}$,
this implies that each entry of $PQ^{-1}$ is contained in $\overline{\mathbb{Q}(p')}$.
On the other hand, since $p$ is generic, the set of coordinates  of $p(v)$ (and hence those of $x_v$) is algebraically independent over $\overline{\mathbb{Q}(p')}$.
Therefore, by regarding the left-hand side of (\ref{eq:algebraic}) as a polynomial in $x_v$,
the polynomial must be identically zero. In particular, we get
$$I_d-PQ^{-1}(PQ^{-1})^T=0.$$
Thus, $PQ^{-1}$ is orthogonal.
In other words, there is some orthogonal matrix $S$ such that $P=SQ$,
and we get $\|p(v_i)+L(\gamma_i)-p(v_0)\|=\|x_i\|=\|Sy_i\|=\|y_i\|=\|q(v_i)+L(\gamma_i)-q(v_0)\|$ for every $1\leq i\leq d$.
Therefore, $q'\in f_{G_v, L}^{-1}( f_{G_v, L}(p'))$.
Since  $(G_v, \psi_v,p)$ is $L$-periodically globally rigid, this in turn implies that
$f_{V-v, L}(p')=f_{V-v, L}(q')$.
Thus we have $p'=q'$.

Since $\{p(v_i)+L(\gamma_i): 0\leq i\leq d\}$ affinely spans $\mathbb{R}^d$, there is a unique extension of $p':V-v\rightarrow \mathbb{R}^d$ to $r:V\rightarrow \mathbb{R}^d$ such that $f_{G, L}(r)=f_{G,L}(p)$.
Thus we obtain $p=q$.
\end{proof}

Note that it follows from \cite[Lemma 3.1]{kst} that a generic $\Gamma$-labelled framework $(G, \psi, p)$ in $\mathbb{R}^d$
with $|V(G)|\geq 2$, rank $k$ periodicity $\Gamma$, and nonsingular lattice $L:\Gamma\rightarrow \mathbb{R}^d$ cannot be $L$-periodically globally rigid in $\mathbb{R}^d$ if the rank of $(G,\psi)$ is less than $k$. This is because in this case the covering $(\tilde G, \tilde p)$ of $(G,\psi,p)$ has infinitely many connected components, each of which may be `flipped' individually in a periodic fashion to obtain an $L$-periodic framework $(\tilde G, \tilde q)$ which is equivalent, but not congruent to $(\tilde G, \tilde p)$. 

This is illustrated by the two equivalent but non-congruent $2$-periodic frameworks in $\mathbb{R}^2$ shown in Figure~\ref{fig:flip} whose $\Gamma$-labelled graph $(G,\psi)$ has rank 1. Note, however, that $(G,\psi)$ is $L$-periodically $2$-rigid, since it is $L$-periodically rigid and the removal of any vertex results in a trivial framework with one vertex orbit and no edges (recall also Theorem~\ref{prop:rigidity_matrix}). 

\begin{figure}[htp]
\begin{center}
\begin{tikzpicture}[thick,scale=1]
\tikzstyle{every node}=[circle, draw=black, fill=black, inner sep=0pt, minimum width=3pt];

\node (pa) at (0,0) {};
\node[draw=black,fill=white](pb) at (0.9,0.6) {};

 \node [rectangle,draw=white, fill=white] (c) at (-0.1,0.6) {\tiny$(0,0)$};
 \node [rectangle,draw=white, fill=white] (c) at (1,0) {\tiny$(1,0)$};
 
\path(pa) edge [->,bend right=40] (pb);
\path(pa) edge [->,bend left=40] (pb);

\node [rectangle,draw=white, fill=white] (c) at (0,-0.5) {$\qquad$};

\end{tikzpicture}
\hspace{1cm}
\begin{tikzpicture}[thick,scale=1]
\tikzstyle{every node}=[circle, draw=black, fill=black, inner sep=0pt, minimum width=3pt];

\draw[dashed, thin] (0,0)--(5,0);
\draw[dashed, thin] (0,1)--(5,1);
\draw[dashed, thin] (0,2)--(5,2);

\draw[dashed, thin] (1,-0.3)--(1,2.3);
\draw[dashed, thin] (2,-0.3)--(2,2.3);
\draw[dashed, thin] (3,-0.3)--(3,2.3);
\draw[dashed, thin] (4,-0.3)--(4,2.3);

\node (p1) at (0.3,0.4) {};
\node (p1r) at (1.3,0.4) {};
\node (p1rr) at (2.3,0.4) {};
\node (p1rrr) at (3.3,0.4) {};
\node (p1rrrr) at (4.3,0.4) {};

\node (p1o) at (0.3,1.4) {};
\node (p1or) at (1.3,1.4) {};
\node (p1orr) at (2.3,1.4) {};
\node (p1orrr) at (3.3,1.4) {};
\node (p1orrrr) at (4.3,1.4) {};

\node [draw=black,fill=white](p2) at (0.7,0.6) {};
\node [draw=black,fill=white](p2r) at (1.7,0.6) {};
\node [draw=black,fill=white](p2rr) at (2.7,0.6) {};
\node [draw=black,fill=white](p2rrr) at (3.7,0.6) {};
\node [draw=black,fill=white](p2rrrr) at (4.7,0.6) {};

\node[draw=black,fill=white] (p2o) at (0.7,1.6) {};
\node[draw=black,fill=white] (p2or) at (1.7,1.6) {};
\node[draw=black,fill=white] (p2orr) at (2.7,1.6) {};
\node[draw=black,fill=white] (p2orrr) at (3.7,1.6) {};
\node[draw=black,fill=white] (p2orrrr) at (4.7,1.6) {};

\draw(p2)--(0,0.5);
\draw(p2o)--(0,1.5);

\draw(p1rrrr)--(5,0.5);
\draw(p1orrrr)--(5,1.5);

\draw(p1)--(p2);
\draw(p1r)--(p2r);
\draw(p1rr)--(p2rr);
\draw(p1rrr)--(p2rrr);
\draw(p1rrrr)--(p2rrrr);

\draw(p1o)--(p2o);
\draw(p1or)--(p2or);
\draw(p1orr)--(p2orr);
\draw(p1orrr)--(p2orrr);
\draw(p1orrrr)--(p2orrrr);

\draw(p1o)--(p2or);
\draw(p1or)--(p2orr);
\draw(p1orr)--(p2orrr);
\draw(p1orrr)--(p2orrrr);

\draw(p1)--(p2r);
\draw(p1r)--(p2rr);
\draw(p1rr)--(p2rrr);
\draw(p1rrr)--(p2rrrr);
\end{tikzpicture}
\hspace{1cm}
\begin{tikzpicture}[thick,scale=1]
\tikzstyle{every node}=[circle, draw=black, fill=black, inner sep=0pt, minimum width=3pt];

\draw[dashed, thin] (0,0)--(5,0);
\draw[dashed, thin] (0,1)--(5,1);
\draw[dashed, thin] (0,2)--(5,2);

\draw[dashed, thin] (1,-0.3)--(1,2.3);
\draw[dashed, thin] (2,-0.3)--(2,2.3);
\draw[dashed, thin] (3,-0.3)--(3,2.3);
\draw[dashed, thin] (4,-0.3)--(4,2.3);

\node (p1) at (0.3,0.4) {};
\node (p1r) at (1.3,0.4) {};
\node (p1rr) at (2.3,0.4) {};
\node (p1rrr) at (3.3,0.4) {};
\node (p1rrrr) at (4.3,0.4) {};

\node (p1o) at (0.3,1.4) {};
\node (p1or) at (1.3,1.4) {};
\node (p1orr) at (2.3,1.4) {};
\node (p1orrr) at (3.3,1.4) {};
\node (p1orrrr) at (4.3,1.4) {};

\node[draw=black,fill=white] (p2) at (0.7,0.2) {};
\node [draw=black,fill=white](p2r) at (1.7,0.2) {};
\node[draw=black,fill=white] (p2rr) at (2.7,0.2) {};
\node [draw=black,fill=white](p2rrr) at (3.7,0.2) {};
\node[draw=black,fill=white] (p2rrrr) at (4.7,0.2) {};

\node [draw=black,fill=white](p2o) at (0.7,1.2) {};
\node [draw=black,fill=white](p2or) at (1.7,1.2) {};
\node[draw=black,fill=white] (p2orr) at (2.7,1.2) {};
\node [draw=black,fill=white](p2orrr) at (3.7,1.2) {};
\node[draw=black,fill=white] (p2orrrr) at (4.7,1.2) {};

\draw(p2)--(0,0.3);
\draw(p2o)--(0,1.3);

\draw(p1rrrr)--(5,0.3);
\draw(p1orrrr)--(5,1.3);

\draw(p1)--(p2);
\draw(p1r)--(p2r);
\draw(p1rr)--(p2rr);
\draw(p1rrr)--(p2rrr);
\draw(p1rrrr)--(p2rrrr);

\draw(p1o)--(p2o);
\draw(p1or)--(p2or);
\draw(p1orr)--(p2orr);
\draw(p1orrr)--(p2orrr);
\draw(p1orrrr)--(p2orrrr);

\draw(p1o)--(p2or);
\draw(p1or)--(p2orr);
\draw(p1orr)--(p2orrr);
\draw(p1orrr)--(p2orrrr);

\draw(p1)--(p2r);
\draw(p1r)--(p2rr);
\draw(p1rr)--(p2rrr);
\draw(p1rrr)--(p2rrrr);

\end{tikzpicture}
\end{center}
\vspace{-0.2cm}
\caption{Example of a $\mathbb{Z}^2$-labelled graph $(G,\psi)$ with $\rank (G,\psi)=1$ (on the left) and two equivalent but not congruent $L$-periodic frameworks $(\tilde G,\tilde p)$ and $(\tilde G,\tilde q)$ with rank 2 periodicity in $\mathbb{R}^2$. }
\label{fig:flip}
\end{figure}
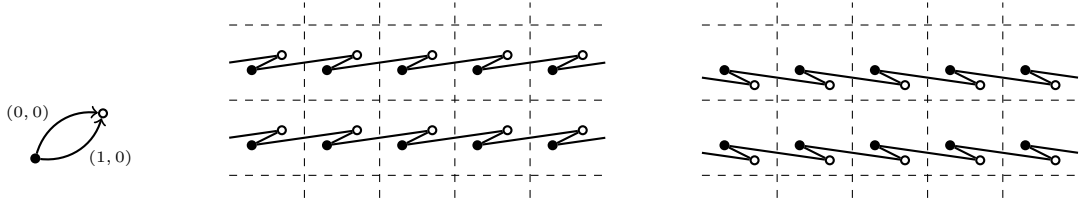

It follows that in the case when $k=d$, $L$-periodic 2-rigidity is not sufficient for $L$-periodic global rigidity. In this case we need the added assumption that $\rank (G,\psi)=d$. In the case when $k<d$ and $\rank (G,\psi)<k$, $(G,\psi, p)$ can also not be $L$-periodically globally rigid, by \cite[Lemma 3.1]{kst}. However, in this case, $(G,\psi, p)$ is also not $L$-periodically 2-rigid.

\begin{theorem}\label{thm:2RGR}
Let $(G, \psi, p)$ be a generic $\Gamma$-labelled framework in $\mathbb{R}^d$ with
rank $k$ periodicity $\Gamma$, and let $L:\Gamma\rightarrow \mathbb{R}^d$ be nonsingular. If $(G, \psi, p)$ is $L$-periodically 2-rigid, and if $(G,\psi)$ is also of rank $d$ in the case when $k=d$, then $(G, \psi, p)$ is also $L$-periodically globally rigid in $\mathbb{R}^d$.
\end{theorem}
\begin{proof}
%
We use induction on \(|V(G)|\). If \(|V(G)|\leq d-k+1\), then $(G, \psi, p)$ is $L$-periodically globally rigid by the $L$-periodic rigidity of $(G, \psi, p)$ and Corollary \ref{cor:smallrigidglobrigid}.
%

Now suppose that \(|V(G)|\geq d-k+2\), and let $(\tilde G, \tilde p)$ be the covering of $(G,\psi, p)$. By our assumption, $(G-v,\psi|_{G-v},p|_{V(G)-v})$ is $L$-periodically rigid for any vertex \(v\in V(G)\).

Suppose first that \(|V(G)|= d-k+2\). Then $(G-v,\psi|_{G-v},p|_{V(G)-v})$ is also $L$-periodically globally rigid by Corollary \ref{cor:smallrigidglobrigid}. We claim that for any occurrence of any \(v\in V(G)\) in the covering $\tilde G$, the affine span of the set $\{\tilde{p}(w) |\, vw \in E(\tilde{G})\}$   is all of \(\mathbb{R}^d\).

 If $d=k$ (and hence $|V(G)|=d-k+2=2$) the claim follows from the fact that  $\rank (G,\psi)=d$, by our assumption. 
 
   If $d>k$ (and hence $|V(G)|=d-k+2>2$), then 
   we suppose for a contradiction that the claim is not true. Then the removal of a neighbour of $v$ (and of all vertices belonging to that same vertex orbit) results in an $L$-periodic framework with at least two distinct orbits of points (since $|V(G)|>2$) and, by our genericity assumption, this framework has the property that all the points connected to $\tilde{p}(v)$ affinely span a space of dimension at most $d-2$, so that $\tilde{p}(v)$  can be rotated about this $(d-2)$-dimensional axis. Since all copies of points in the same orbit can then also be rotated in a periodic fashion and the affine span of the non-moving points is $(d-1)$-dimensional (as a $k$-periodic configuration with $d-k$ vertex orbits), we obtain a contradiction to the $L$-periodic $2$-rigidity of $(\tilde G,\tilde p)$. 
   
   Thus, the affine span of the 
points $\{\tilde{p}(w) |\, vw \in E(\tilde{G})\}$
 is indeed all of \(\mathbb{R}^d\) as claimed, and it follows from Lemma~\ref{lem:redrig} that $(G,\psi, p)$ is $L$-periodically globally rigid.

We may therefore  assume that $|V(G)|>d-k+2$. We show that $(G_v, \psi_v, p')$ is $L$-periodically 2-rigid for any $v\in V(G)$. Suppose for a contradiction that this is not true. Then there is a vertex \(u\) whose removal results in an $L$-periodically flexible framework. As the neighbours of one occurrence of \(v\) in \(\tilde{G}\) induce a complete graph in $\tilde{G_v}$ (where  any pair of vertices from the same vertex orbit may always be considered adjacent due to the fixed lattice representation), adding \(v\) together with its incident edges  to $(G_v-u, \psi_v|_{G_v-u}, p|_{V(G)-\{u,v\}})$ still yields an $L$-periodically flexible framework. This is a contradiction, as $(G-u,\psi|_{G-u}, p|_{V(G)-u})$ is an $L$-periodically rigid $\Gamma$-labelled spanning subframework of the framework obtained from $(G_v-u, \psi_v|_{G_v-u}, p|_{V(G)-\{u,v\}})$ by adding $v$ and its incident edges.

Thus $(G_v, \psi_v, p')$ is $L$-periodically 2-rigid as claimed. Moreover, since $(G,\psi)$ is 2-connected by the $L$-periodic $2$-rigidity of $(G,\psi,p)$, it follows from the definition of $(G_v,\psi_v)$ that $\Gamma_G=\Gamma_{G_v}$. Thus, if $(G,\psi)$ is of rank $d$ then so is $(G_v,\psi_v)$. It now follows from the induction hypothesis that $(G_v, \psi_v, p')$ is $L$-periodically globally rigid. Moreover, by the same argument as above for the case when $|V(G)|=d-k+2>2$, the affine span of the points $\{\tilde{p}(w) |\, vw \in E(\tilde{G})\}$
 is all of \(\mathbb{R}^d\). Thus, by Lemma \ref{lem:redrig}, $(G, \psi, p)$ is $L$-periodically globally rigid.
\end{proof}


\section{Global rigidity of body-bar frameworks}\label{sec:bbglob}

Using Theorem~\ref{thm:2RGR} and the results in  \cite{kst}, we can now easily prove the following extension of Theorem~\ref{thmbbglobal} to periodic body-bar frameworks. We need the following definition.

An $L$-periodic body-bar realisation $(G_{\tilde{H}},\tilde p)$ of a multi-graph $\tilde H$ is {\em $L$-periodically bar-redundantly rigid} if for every edge orbit  $\tilde e$ of $\tilde H$, the  framework $(G_{\tilde{H}}- \tilde {e'}, \tilde p)$ is $L$-periodically rigid. (Recall the definition of a body-bar realisation in Section~\ref{subsec:bb}.)

\begin{theorem}\label{thm:bodybarglobalrig}
Let $(G_{\tilde{H}},\tilde p)$ be a generic $L$-periodic body-bar realisation of the multi-graph $\tilde H$ in $\mathbb{R}^d$ with rank $k$ periodicity $\Gamma$, and let $L:\Gamma\rightarrow \mathbb{R}^d$ be nonsingular.
Then $(G_{\tilde{H}},\tilde p)$ is $L$-periodically globally rigid in \(\mathbb{R}^d\) if and only if $(G_{\tilde{H}},\tilde p)$ is  $L$-periodically bar-redundantly rigid in \(\mathbb{R}^d\), and the quotient $\Gamma$-labelled graph of $G_{\tilde H}$ is of rank $d$ in the case when $k=d$.
\end{theorem}
\begin{proof} It immediately follows from  \cite[Lemma 3.7]{kst} that $L$-periodic bar-redundant  rigidity is necessary for a generic $L$-periodic body-bar realisation to be $L$-periodically globally rigid. 
Moreover, it follows from  \cite[Lemma 3.1]{kst} that in the case when $k=d$,  the rank of the quotient $\Gamma$-labelled graph of $G_{\tilde H}$ must be equal to $d$  for a generic $L$-periodic body-bar realisation to be $L$-periodically globally rigid (recall also the discussion in Section~\ref{sec:2rigid}).
 It is also easy to see that if a generic $L$-periodic body-bar realisation is $L$-periodically bar-redundantly rigid, then it is $L$-periodically 2-rigid, since the edges connecting the bodies are all disjoint. The result now follows from Theorem~\ref{thm:2RGR}.
\end{proof}

Note that generic $L$-periodic bar-redundant rigidity can easily be checked in polynomial time  based on the combinatorial characterisation of generic $L$-periodic rigidity of body-bar frameworks in $\mathbb{R}^d$ conjectured by Ross in \cite[Conjecture 5.1]{rossbb} and proved by Tanigawa in \cite[Theorem 7.2]{T2}. Using our notation and a simplified expression for the dimension of the space of trivial motions for a $k$-periodic framework in $\mathbb{R}^d$, this result may be restated as follows.

\begin{theorem}[\cite{T2}] \label{thm:bodybarrig}
Let $(G_{\tilde{H}},\tilde p)$ be a generic $L$-periodic body-bar realisation of the multi-graph $\tilde H$ in $\mathbb{R}^d$ with rank $k$ periodicity $\Gamma$, and $L:\Gamma\rightarrow \mathbb{R}^d$ be nonsingular. Then $(G_{\tilde{H}},\tilde p)$ is $L$-periodically rigid in $\mathbb{R}^d$ if and only if the quotient $\Gamma$-labelled graph $H$ of $\tilde{H}$  contains a spanning subgraph $(V,E)$ satisfying the following counts
\begin{itemize}
\item $|E|=\binom{d+1}{2}|V|-d-\binom{d-k}{2}$;
\item $|F|\leq \binom{d+1}{2}|V(F)|-d-\binom{d-k(F)}{2}$ $\quad$ for  all non-empty $F\subseteq E$,
\end{itemize}
where $k(F)$ is the rank of  $F$.
\end{theorem}

\section{Conclusion and further comments}

Real-world structures, whether they are natural such as crystals or proteins, or man-made such as buildings or linkages, are usually non-generic, and often exhibit non-trivial symmetries. This fact has motivated a significant amount of research in recent years on how symmetry impacts the rigidity and flexibility of frameworks (see \cite{HoDCG}, for example, for a summary of results). In Theorem \ref{thm:2RGR}, we have shown that the sufficient condition given by Tanigawa in \cite{T1} for generic global rigidity of finite frameworks can be transformed to a sufficient condition for generic global rigidity of infinite $L$-periodic frameworks (under a fixed lattice $L$). It remains open whether this result can be extended to other types of frameworks with symmetries such as infinite periodic frameworks with (partially) flexible lattices or finite frameworks with point group symmetries.
Following the proof of Theorem \ref{thm:bodybarglobalrig}, such an extension would imply the characterisation of the generic global rigidity of finite body-bar frameworks with these symmetries by using the existing (local) rigidity characterisations of these frameworks by Tanigawa \cite{T2}. Furthermore, such a result would be useful for the characterisation of the generic global rigidity of \emph{body-hinge frameworks}  with symmetries (where the bodies are connected in pairs by $d-2$-dimensional hinges) such as in the (finite) generic version established by Jord\'{a}n, Kir\'{a}ly and Tanigawa \cite{jkt2}. However, the characterisation of generic (local) rigidity for periodic body-hinge frameworks is still open (even for fixed lattices). For finite symmetric body-hinge frameworks, such a characterisation is only known for groups of the form $\mathbb{Z}_2\times \mathbb{Z}_2\times \cdots \times \mathbb{Z}_2$ \cite{ST14}.
 A major goal in this research area is to obtain a combinatorial characterisation of the generic global rigidity of infinite $L$-periodic or finite symmetric \emph{molecular frameworks} in 3-space (i.e., body-hinge frameworks in 3-space with the added property that the lines of the hinges attached to each body all meet in a single point on that body), since they may be used to model crystals and protein structures. We note that for finite molecular frameworks,  their generic (local) rigidity was recently characterised by the celebrated result of Katoh and Tanigawa \cite{molecular}. However, their generic global rigidity has not yet been characterised, and there are also no generic local or global rigidity characterisations for infinite $L$-periodic or finite symmetric molecular frameworks \cite{portaetal}.

\section{Acknowledgements}

  The first author was supported by the Hungarian Scientific Research Fund of the National Research, Development and Innovation Office (OTKA, grant number K109240 and K124171). The second author was supported by the \'UNKP-17- 4 New National Excellence Program of the Ministry of Human Capacities of Hungary and by the Hungarian Scientific Research Fund of the National Research, Development and Innovation Office (OTKA, grant number K109240). The third author was supported by EPSRC First Grant EP/M013642/1.

\end{document}